\documentclass[10pt]{amsart}%
\usepackage{amsmath}
\usepackage{amscd}
\usepackage{amsfonts}
\usepackage{amssymb}%
\newtheorem{theorem}{Theorem}[section]

\newtheorem{proposition}[theorem]{Proposition}

\newtheorem{definition}[theorem]{Definition}
\newtheorem{corollary}[theorem]{Corollary}
\newtheorem{remark}[theorem]{Remark}
\newcommand{\be}{\begin{equation}}
\newcommand{\ee}{\end{equation}}
\newcommand{\bes}{\begin{equation*}}
\newcommand{\ees}{\end{equation*}}

\newcommand{\cD}{\mathcal{D}}

\newcommand{\mb}[1]{\mathbb{#1}}

\begin{document}

\title[Dilation theory in finite dimensions]{Dilation theory in finite dimensions: the possible, the impossible and the unknown}

\author{Eliahu Levy}
\address{Dept. of Mathematics\\
Technion IIT\\
Haifa 32000, Israel}
\email{eliahu@tx.technion.ac.il}
\author{Orr Moshe Shalit}
\address{Pure Mathematics Dept.\\
University of Waterloo\\
Waterloo, ON\; N2L--3G1\\
Canada}
\email{oshalit@uwaterloo.ca}
\date{}
\subjclass[2010]{47A20 (Primary), 15A45, 47A57 (Secondary)}
\maketitle

\begin{abstract}
This expository essay discusses a finite dimensional approach to dilation theory. How much of dilation theory can be worked out within the realm of linear algebra? It turns out that some interesting and simple results can be obtained. These results can be used to give very elementary proofs of sharpened versions of some von Neumann type inequalities, as well as some other striking consequences about polynomials and matrices.  Exploring the limits of the finite dimensional approach sheds light on the difference between those techniques and phenomena in operator theory that are inherently infinite dimensional, and those that are not.
\end{abstract}

\section{A single contraction}

In these notes $H$ is a finite dimensional real or complex Hilbert space of dimension $\dim H = n$. A contraction is an operator $T$ with  $\|T\|\leq 1$.

Von Neumann's Inequality states that for every contraction $T$ on a Hilbert space  and every polynomial $p$,
\be\label{eq:vN}
\|p(T)\| \leq \|p\|_\infty :=  \sup_{|z|=1} |p(z)|.
\ee
There is an elegant proof of this inequality using Sz.-Nagy's Dilation Theorem:
\begin{theorem}\label{thm:NagyDT}{\bf Sz.-Nagy's Dilation Theorem \cite[Thm. 4.2]{SzNF70}.}
Let $T$ be a contraction acting on a Hilbert space $H$. Then there exists a Hilbert space $K \supseteq H$ and a unitary $U$ on $K$ such that
\be\label{eq:dilation}
T^k = P_H U^k P_H \,\, , \,\, k \in \mb{N}.
\ee
\end{theorem}
Here and below, $P_H$ denotes the projection of $K$ onto $H$.

The elegant proof alluded to runs as follows (see \cite[Sec. 8]{SzNF70}): by the Spectral Theorem, von Neumann's Inequality holds for any unitary. By virtue of (\ref{eq:dilation}), $p(T) = P_H p(U) P_H$ for any polynomial. Thus 
\bes
\|p(T)\| \leq \|p(U)\| \leq \|p\|_\infty.
\ees

When $T$ is not a unitary, then the space $K$ in the above theorem is necessarily infinite dimensional. Note that (\ref{eq:vN}) is not a trivial fact even when $H$ is finite dimensional. A motivation for writing these notes was to find a proof of (\ref{eq:vN}) and its generalizations for finite dimensional spaces that does not involve infinite dimensional Hilbert space, but which does involve dilation theory. There do exist proofs of this inequality for when $T$ is a matrix (see Chapter 1 in \cite{Pisier}, or Exercise 2.16 in \cite{Paulsen}), but we believe that the one presented below is the most elementary. It is probably not new, but in our opinion it should be recorded.

It was mentioned above that $K$ in Theorem \ref{thm:NagyDT} is always infinite dimensional when $T$ is not unitary. On the other hand, Halmos noticed \cite{Halmos50} that one can dilate $T$ to a unitary which acts on $H \oplus H$, given by
\be\label{eq:Halmos}
U = \begin{pmatrix}
T & (I - T T^*)^{1/2}\\
(I - T^* T)^{1/2} & -T^*
\end{pmatrix} .
\ee
However, this $U$ satisfies (\ref{eq:dilation}) only for $k=1$. This motivates the following definition.
\begin{definition}
Let $T$ be an operator on $H$, and let $N \in \mb{N}$. A \emph{unitary $N$-dilation for $T$} is a unitary $U$ acting on $K$ such that (\ref{eq:dilation}) holds for all $k = 1, \ldots, N$.
\end{definition}
Remarkably, even $1$-dilations do have applications \cite{Halmos50}; see \cite{ChoiLi} for a relatively recent one.

We make some standard definitions. 
As usual, a dilation will be called \emph{minimal} if the smallest subspace $L \subseteq K$ such that $H\subseteq L$ and $UL = L$, is $K$ itself. It will be called \emph{$N$-minimal} if $K = \textrm{span}\{U^k h : h \in H, k = 0, \ldots, N\}$. Two $N$-dilations $U_1,K_1$ and $U_2,K_2$ will be called \emph{isomorphic} if there is a unitary $W: K_1 \rightarrow K_2$ that fixes $H$ and intertwines $U_1$ and $U_2$.

Henceforth $H$ is a finite dimensional Hilbert space, $\dim H = n$. For a contraction $T$ on $H$ we  define $D_T = (I-T^*T)^{1/2}$, $\cD_T = \textrm{Im}(D_T)$, and $d_T = \dim \cD_T$. Note that $\dim \cD_{T^*} = d_T$.

\begin{theorem}\label{thm:Ndil}
For every $N$, every contraction $T$ on $H$ has an $N$-minimal, unitary $N$-dilation acting on a space of dimension $n + Nd_T$. This dilation is also minimal.
\end{theorem}
\begin{proof}
Let $V$ be a unitary from $\cD_{T}$ onto $\cD_{T^*}$. On the space 
\bes
K = H \oplus \underbrace{\cD_{T^*} \oplus \cdots \oplus \cD_{T^*}}_{N \textrm{ times}}
\ees
we define
\bes
U = \begin{pmatrix}
T &  & & & D_{T^*} \\
V D_T & & &  & - VT^* \\
 & I &  & &\\
 &   & \ddots & & \\
 &   & & I & 0
\end{pmatrix} .
\ees
The empty slots are understood as $0$'s, and the sub-diagonal dots are all $I$'s.
Using the identity $T D_T = D_{T^*}T$, we find that $U^*U = I$, thus $U U^* = I$. Eq. (\ref{eq:dilation}) for $k \leq N$ is verified mechanically. The $N$-minimality is obvious, and minimality follows. 
\end{proof}

\begin{remark}\emph{
Taking the inductive limit of the above dilations in an obvious way, or, equivalently, the strong operator limit (once assembled inside an infinite dimensional space), one obtains the minimal \emph{isometric} dilation of Sz.-Nagy.}
\end{remark}

The elementary proof of von Neumann's Inequality for operators on a finite dimensional space is the same as the one we gave above, with the following changes. Given $T$ and a polynomial $p$ of degree $N$, construct a unitary $N$-dilation $U$ for $T$, acting on a space of dimension $M = n + Nd_T$  (this takes the place of constructing the full unitary dilation). Instead of invoking the spectral theorem for normal operators on Hilbert space, we use the fact that every unitary matrix is unitarily diagonalizable. Thus we may assume that $U = \textrm{diag}(\lambda_1, \ldots, \lambda_{M})$. Now, $p(U) = \textrm{diag}(p(\lambda_1) \ldots, p(\lambda_M))$, hence $\|p(U)\| = \sup_i |p(\lambda_i)| \leq \|p\|_\infty$ (see Theorem \ref{thm:vNsharp} below for a sharpening of this result). 

It should also be mentioned that von Neumann's Inequality for operators on an infinite dimensional space follows from the finite dimensional case, see \cite[p. 15]{Pisier}.

According to Sz.-Nagy \cite[Section 4]{SzNapp}, Theorem \ref{thm:Ndil}  was first published in 1954 by Egerv\'ary \cite{Egervary} (more or less: the dilation in \cite{Egervary} is not minimal). It is clear that Halmos thought of this construction, but did not think that it is a big deal (evidence: the discussion in \cite[Problem 227]{HalmosBook}). It seems that the idea of $N$-dilations was abandoned. One possible reason is that Sz.-Nagy's paper \cite{SzNDil} on the existence of a unitary dilation appeared already in 1953. Another reason might be that confining dilation theory to finite dimensions makes it a difficult subject. The rest of these notes contains further discussion of this finite dimensional approach to dilation theory. We will present what we know about this topic, alongside some neat results that seem to have been overlooked as well as open problems.

\section{A little more on minimality}

There are many minimal $N$-dilations which are not isomorphic. This can be seen by taking $T = 0 \in \mb{C}$, and considering the two minimal $1$-dilations given by eq. (\ref{eq:Halmos}), on the one hand, and by the $2$-dilation constructed above, on the other. Even under the assumption of $N$-minimality, or that the dilation act on a space of dimension $n+N d_T$, the dilation is not unique. Consider, for example, the two non-isomorphic $1$-dilations of $T = 0 \in \mb{C}$ given by
\bes
U_1 = \begin{pmatrix} 0 & 1 \\ 1 & 0\end{pmatrix} \,\, \textrm{  \emph{and}  } \,\, U_1 = \begin{pmatrix} 0 & -1 \\ 1 & 0\end{pmatrix} .
\ees
The following proposition says that nonetheless, to a certain extent, minimal dilations look pretty much the same.
\begin{proposition}
Let $T$ be a contraction on $H$, and let $V$ on $L$ be a unitary $N$-dilation for $T$. Let $U$ on $K$ be an $N$-minimal $N$-dilation of $T$. Then there exits an isometry $W: K \rightarrow L$ such that $W\big|_H = I_H$ and 
\be\label{eq:almostint}
V W g = W U g ,
\ee
for all $g \in \textrm{\emph{span}}\{U^k h : h \in H, k = 0, \ldots, N-1\}$.
\end{proposition}
\begin{proof}
We define $W U^k h = V^k h$ for $k = 0,1,\ldots, N$ and $h \in H$. This is a well defined isometry since $U$ is $N$-minimal and since for $h,g \in H$, $0 \leq k \leq m\leq N$,
\begin{align*}
\langle W U^k h, WU^m g \rangle &= \langle V^k h, V^m g \rangle \\
&= \langle h,V^{m-k}g \rangle \\
&= \langle h,T^{m-k}g \rangle \\
&= \langle U^k h, U^m g \rangle .
\end{align*}
Eq. (\ref{eq:almostint}) follows.
\end{proof}
Thus, all $4$-dilations of $T$ essentially have the form
\bes
\begin{pmatrix}
T     &  &  &  & * \\
D_T &  &  &  & * \\
     & I &   &  &  \\
     &  & I  &  &  \\
     &  &   & I &  \\
     &  &   &  & * 
\end{pmatrix},
\ees
(the dimensions of the last row and column are not necessarily $d_T$), and a similar statement can be made for $N$-dilations. In particular, we have the following corollary.
\begin{corollary}
If $T$ acts on $H$, the minimal dimension on which a unitary $N$-dilation of $T$ can act is $n+ N d_T$. A unitary $N$-dilation is $N$-minimal if and only if it acts on a space of dimension $n+Nd_T$.
\end{corollary}

\section{An additional application and a non-application}

In his seminal paper \cite{SzNDil}, Sz.-Nagy presented three applications of his theorem on unitary dilations. We already discussed one application: the simple and elegant proof of von Neumann's Inequality. Another application given in that paper is a simple proof of a simple fact: if $T$ is a contraction and $h$ is an invariant vector for $T$, then $h$ is also invariant for $T^*$. The proof via dilations is as follows. $P_H U h = Th = h$. It follows that $Uh = h$. Thus $U^*h = h$, hence $T^*h = P_H U^* h = h$. Of course, no one would have any problem proving this fact without recourse to dilations, but you have to admit that this is elegant. Note that one only uses the fact that there exists a $1$-unitary dilation for $T$.

As another application of his dilation theorem, Sz.-Nagy gave a very neat proof of the following ergodic theorem:
\begin{theorem}Let $T$ be a contraction. Then
\be\label{eq:ergodic}
\lim_{N \rightarrow \infty} \frac{1}{N+1}\sum_{k=0}^N T^k
\ee
exists in the strong operator topology.
\end{theorem}
\begin{proof}
Let $U$ be any unitary dilation of $T$. Then 
\bes
\frac{1}{N+1}\sum_{k=0}^N T^k = P_H \left(\frac{1}{N+1}\sum_{k=0}^N U^k \right)P_H,
\ees
so it suffices to prove the theorem for the special case where $T$ is a unitary. 
But the result for unitaries is the classical and well known Mean Ergodic Theorem of von Neumann (see \cite[Problem 228]{HalmosBook} for a proof and additional references).
\end{proof}

This application was brought in to show how a dilation that works for all powers can be infinitely more powerful than infinitely many $N$-dilations. Indeed, trying to imitate the above proof, but using $N$-dilations instead of dilations, we come to the expressions
\bes
\frac{1}{N+1}\sum_{k=0}^N T^k = P_H \left(\frac{1}{N+1}\sum_{k=0}^N U_N^k \right)P_H,
\ees
where $U_N$ denotes a unitary $N$-dilation for $T$. One cannot use the Mean Ergodic Theorem (for unitary operators) because the rate of convergence is different for different unitaries. As an illustration, let $T = 0 \in \mb{C}$, and let us choose for $U_N$ the $2N+1$ dilation constructed in Theorem 
\ref{thm:Ndil}. Then $U_N$ is the $(2N+2) \times (2N + 2)$ matrix
\bes
U_N = \begin{pmatrix}
0 &  & & & 1 \\
1 & & &  & 0 \\
 & 1 &  & &\\
 &   & \ddots & & \\
 &   & & 1 & 0
\end{pmatrix} .
\ees
$U_N$ contains the scalar operator $u = \exp{\frac{2\pi i}{2N+2}}$ as a direct summand. But
\bes
\frac{1}{N+1}\sum_{k=0}^N u^k = \frac{2}{(N+1)(1-u)} \sim \frac{2(2N+2)}{2\pi i(N+1)} \sim \frac{2}{\pi i}.
\ees
This is far from the limit, which is $0$. If we would have chosen, in \emph{this} example, $U_N$ to be the $N$-minimal dilation constructed in Theorem \ref{thm:Ndil}, then everything would have worked out, but this is just a coincidence.

\section{Several commuting contractions}

\subsection{Dilations and von Neumann's inequality}

We remind the reader that throughout $H$ is a Hilbert space of finite dimension $n$. 

\begin{definition}
Let $T_1,\ldots,T_k$ be commuting contractions on $H$, and let $N \in \mb{N}$. A \emph{unitary $N$-dilation for $T_1, \ldots, T_k$} is a $k$-tuple of commuting unitaries $U_1,\ldots,U_k$ acting on a space $K\supseteq H$ such that
\be\label{eq:multidil}
T_1^{n_1} \cdots T_k^{n_k} = P_H U_1^{n_1} \cdots U_k^{n_k} P_H ,
\ee
for all  $n_1, \ldots, n_k$ satisfying $n_1 + \ldots + n_k \leq N$.
\end{definition}
The usual definition of unitary dilation of a $k$-tuple can now be phrased as follows: a \emph{unitary dilation} for $T_1,\ldots,T_k$ is a $k$-tuple $U_1,\ldots,U_k$ that is a unitary $N$-dilation for all $N \in \mb{N}$. If one of the $T_i$'s is not a unitary, then a unitary dilation, if it exists, must operate on a space $K$ of infinite dimension. 

Over the years, several conditions that guarantee the existence of a unitary dilation for a $k$-tuple of commuting contractions have been studied (see, e.g., Chapter I of \cite{SzNF70}, \cite{Opela}, \cite{StoSza} or \cite{CiStoSza} and the references within\footnote{Perhaps before anything else the reader might like to see  \cite{Arv10}.}). One of the simplest is the following.

\begin{definition}
Let $T_1,\ldots,T_k$ be operators on $H$. $T_1,\ldots,T_k$ are said to \emph{doubly commute} if for all $i \neq j$, $T_i$ commutes with $T_j$ and with $T_j^*$. 
\end{definition}
\begin{theorem}\label{thm:doublycommuting}
Let $T_1,\ldots,T_k$ be a $k$-tuple of doubly commuting contractions on $H$. Then for every $N$, the $k$-tuple $T_1,\ldots,T_k$ has a unitary $N$-dilation that acts on a space of dimension $(N+1)^kn$.
\end{theorem}
\begin{proof}
We prove the theorem for $N=3$, the general case is proved with minor changes of notation.
On $L = H \oplus H \oplus H \oplus H$, let
\bes
V_1 = \begin{pmatrix}
T_1     &  &  &  D_{T_1^*} \\
D_{T_1} &  &  &  -T_1^* \\
     & I &   &    \\
     &  & I  &   
\end{pmatrix},
\ees
and for $i>1$ we define
\bes
V_i = \begin{pmatrix}
T_i     &  &  &   \\
 & T_i  &  &   \\
     &  & T_i   &    \\
     &  &   & T_i   
\end{pmatrix}.
\ees
Then $V_i$ is a contractive $3$-dilation of $T_i$, and $V_1$ is unitary. Furthermore, if $T_i$ is already unitary, so is $V_i$. Since $T_1,\ldots,T_k$ doubly commute, so do $V_1, \ldots, V_k$. Now on $L \oplus L \oplus L \oplus L$ we define
\bes
W_2 = \begin{pmatrix}
V_2     &  &  &  D_{V_2^*} \\
D_{V_2} &  &  &  -T_2^* \\
     & I &   &    \\
     &  & I  &   
\end{pmatrix},
\ees
and for $i\neq 2$ we define
\bes
W_i = \begin{pmatrix}
V_i     &  &  &   \\
 & V_i  &  &   \\
     &  & V_i   &    \\
     &  &   & V_i   
\end{pmatrix}.
\ees
After carrying out this step $k$ times, we obtain a $k$-tuple $U_1, \ldots, U_k$ which is a unitary $3$-dilation for $T_1,\ldots,T_k$.
\end{proof}

A motivation for studying unitary $N$-dilations is the following sharpening of von Neumann's inequality.
\begin{theorem}\label{thm:vNsharp}
Let $N \in \mb{N}$, and let $T_1,\ldots,T_k$ be $k$-tuple of commuting contractions on $H$ that has a unitary $N$-dilation acting on a finite dimensional Hilbert space $K$. Put $m=\dim K$. Then there exist $m$ points $\{w^i = (w_1^i, \ldots, w_k^i)\}_{i=1}^{m}$ on the $k$-torus $\mb{T}^k$ such that for every polynomial $p(z_1, \ldots, z_k)$ of degree less than or equal to  $N$, 
\bes
\|p(T_1, \ldots, T_k)\| \leq \max \{|p(w^i)| : i=1, \ldots, m\}.
\ees
In particular,
\bes
\|p(T_1, \ldots, T_k)\| \leq \|p\|_\infty := \sup \{|p(z_1, \ldots, z_k)| : |z_i|=1, i=1, \ldots, k\}.
\ees
\end{theorem}
\begin{proof}
Let $U_1, \ldots, U_k$ be a unitary $N$-dilation of $T_1,\ldots,T_k$ on a finite dimensional space $K$. As $U_1, \ldots, U_k$ are commuting unitaries, they are simultaneously unitarily diagonalizable. 
We may assume that $U_j$, $j=1, \ldots, k$, has the form 
\bes
U_j = \begin{pmatrix}
w^1_j     &  &  &   \\
 & w^2_j  &  &   \\
     &  & \ddots  &    \\
     &  &   &  w^{m}_j  
\end{pmatrix}.
\ees
Thus,
\bes
\|p(T_1, \ldots, T_k)\| = \|P_Hp(U_1, \ldots, U_k) P_H\| \leq  \left\|\begin{pmatrix}
p(w^1)     &  &  &   \\
 & p(w^2)  &  &   \\
     &  & \ddots  &    \\
     &  &   &  p(w^{m})  
\end{pmatrix} \right\|,
\ees
and the right hand side is equal to $\max \{|p(w^i)| : i=1, \ldots, m\}$.
\end{proof}

Every contraction has an $N$-dilation, so the above theorem is an interesting sharpening of von Neumann's inequality (for the expert it might be interesting to compare this sharpening with that provided in \cite{AM05}). In fact, even for the case of {\em scalar} operators it is non-trivial. 

\begin{corollary}\label{cor:scalar_vN}
Let $\zeta$ be a point in the unit disc $\mb{D} = \{z \in \mb{C} : |z|<1\}$, and let $N$ be a positive integer. Then there exist $N+1$ points $w_0, \ldots, w_N$ on the unit circle $\mb{T} = \{z \in \mb{C}: |z|=1\}$ such that for every polynomial $p$ of degree less than or equal to $N$, one has
\bes
|p(\zeta)| \leq \max_i |p(w_i)|.
\ees
\end{corollary}

Following an example of Kaijser and Varopoulos \cite{Varopoulos}, Holbrook showed \cite{Hol99} that there exist three $4\times 4$ commuting matrices $A_1, A_2, A_3$ with $\|A_i\| \leq 1$ such that for the polynomial $p(x,y,z) = x^2 + y^2 + z^2 - 2xy -2xz -2yz$ one has
\bes
\|p(A_1, A_2, A_3)\| = \frac{6}{5}\|p\|_\infty.
\ees
This shows that even a $2$-dilation may not exist for three commuting contractions, regardless of the restrictions we put on the dimensionality of the dilation space $K$. In fact, Parrott \cite{Parrott} constructed an example of three commuting operators (acting on a four dimensional space) for which there is no $1$-dilation (see also \cite[p. 909]{HalmosTen}). The proof that Parrott's example is not a $1$-dilation does not involve violation of von Neumann's Inequality.

It is interesting to note that  von Neumann's inequality holds for any $k$-tuple of commuting $2\times 2$ contractions (see \cite[p. 21]{Drury} or \cite{Hol90}). Whether or not this is true for $3 \times 3$ matrices is an open problem. 

\vspace{0.2cm}
\noindent {\bf Problem A:}
\emph{Do there exist $k$ commuting $3\times 3$ matrices $A_1, \ldots, A_k$ with $\|A_i\| \leq 1$ for $i=1,\ldots,k$, such that there is some polynomial $p$ in $k$ variables for which $\|p(A_1, \ldots, A_k)\| > \|p\|_\infty$ ?}
\vspace{0.2cm}

In \cite{Hol99} it is mentioned that evidence suggests that the answer is negative. It is rather humbling to know that even the case $k=3$ is still open.

\subsection{Regular dilations}

For $m = (m_1, \ldots, m_k) \in \mb{Z}^k$, let us define
\bes
T(m) =  (T_1^{{m_1}_-} \cdots T_k^{{m_k}_-})^* T_1^{{m_1}_+} \cdots T_k^{{m_k}_+},
\ees
where we use the usual notation of positive and negative parts of a number: $x_+ = \max\{x,0\}$ and $x_- = x_+ - x$. 
Examining the proof of Theorem \ref{thm:doublycommuting}, we find that the unitary $N$-dilation $U_1,\ldots,U_k$ for $T_1,\ldots,T_k$ satisfies 
\be\label{eq:regulardil}
T(m) = P_H U(m) P_H 
\ee
 for every $m \in \mb{Z}^k$ such that $|m| := |m_1| + \ldots + |m_k| \leq N$, which is much stronger than (\ref{eq:multidil}). Indeed, returning to the notation of the proof, it is easy to see that $T(m) = P_H V(m) P_H$ for all $m \in \mb{Z}^k$ with $|m|\leq N$. Similarly, $P_H W(m) P_H = P_H P_L W(m) P_L P_H = P_H V(m) P_H = T(m)$. Continuing this way, we find that $P_H U(m) P_H =T(m)$. 

A unitary dilation satisfying (\ref{eq:regulardil}) for all $m \in \mb{Z}^k$ is said to be a \emph{regular} dilation. Following this terminology, we will call a $k$-tuple $U_1, \ldots, U_k$ of commuting unitaries a \emph{regular $N$-dilation} for $T_1, \ldots, T_k$ if (\ref{eq:regulardil}) holds for all $m \in \mb{Z}^k$ with $|m|\leq N$. 

\begin{remark}\emph{
Note that if one replaces the unitary $N$-dilations appearing in the proof of Theorem \ref{thm:doublycommuting} by the minimal isometric Sz.-Nagy dilations, then one gets a proof that every $k$-tuple of doubly commuting contractions has a regular, doubly commuting isometric dilation. To our knowledge, this proof is new.}
\end{remark}

There are several known conditions that ensure that a $k$-tuple of contractions has a regular dilation, and one of them is that the $k$-tuple doubly commute (see Section I.9, \cite{SzNF70}). A necessary and sufficient condition for a a $k$-tuple $T_1, \ldots, T_k$ to have a regular unitary dilation is that for all $u \subseteq \{1, \ldots, k\}$, we have the operator inequality
\be\label{eq:oi}
\sum_{v \subseteq u} (-1)^{|v|}T(e(v))^*T(e(v)) \geq 0,
\ee
where $e(v) \in \{0,1\}^k$ is $k$-tuple that has $1$ in the $i$th slot if and only if $i \in v$ \cite[Theorem I.9.1]{SzNF70}. This leads us to ask the following question.

\vspace{0.2cm}
\noindent {\bf Problem B:}
\emph{For $k$-tuples $T_1,\ldots,T_k$ of commuting contractions on $H$, are the following two conditions equivalent?
\begin{enumerate}
\item For every $N \in \mb{N}$, $T_1,\ldots,T_k$ has a regular unitary $N$-dilation on a finite dimensional Hilbert space $K$.
\item Inequality (\ref{eq:oi}) holds for all $u \subseteq \{1, \ldots, k\}$.
\end{enumerate}}

\subsection{Further consequences of $N$-dilations}

Using $N$-dilations one can obtain the following more precise result, which implies Theorem \ref{thm:vNsharp} immediately.
\begin{theorem}\label{thm:convex_sum}
Let $N \in \mb{N}$, and let $T_1,\ldots,T_k$ be $k$-tuple of commuting contractions on $H$ that has a unitary $N$-dilation acting on a finite dimensional Hilbert space $K$, with $\dim K = m$. Then there exist $m$ points $\{w^i = (w_1^i, \ldots, w_k^i)\}_{i=1}^{m}$ on the $k$-torus $\mb{T}^k$, and $m$ positive operators $A_1 \ldots, A_{m} \in B(H)$ satisfying $\sum A_i = I_H$, such that for every polynomial $p(z_1, \ldots, z_k)$ of degree less than or equal to  $N$, 
\be\label{eq:convex_sum}
p(T_1, \ldots, T_k) = \sum_{i=1}^{m} p(w^i_1, \ldots, w^i_k) A_i.
\ee
\end{theorem}
\begin{proof}
Using the notation of the proof of Theorem \ref{thm:vNsharp}, we have that $U_1, \ldots, U_k$ are all diagonal with respect to an orthonormal basis $\{e_1, \ldots, e_m\}$. Therefore, for all $j=1, \ldots, k$,
\bes
U_j = \sum_{i=1}^m w_j^i e_i e_i^*
\ees
and so for any polynomial $p$, $\deg p \leq N$,
\bes
p(U_1, \ldots, U_k) = \sum_{i=1}^m p(w_1^i, \ldots, w_k^i) e_i e_i^* .
\ees
Since $U_1, \ldots, U_k$ is an $N$-dilation for $T_1, \ldots, T_k$, we have
\bes
p(T_1, \ldots, T_k) = P_H p(U_1, \ldots, U_k) P_H =   \sum_{i=1}^m p(w_1^i, \ldots, w_k^i)  P_H e_i e_i^* P_H .
\ees
This gives (\ref{eq:convex_sum}) with $A_i = (P_H e_i) (P_H e_i)^*$.
\end{proof}

Theorems \ref{thm:doublycommuting} and \ref{thm:convex_sum} together give the following interesting result about polynomials:
\begin{corollary}
Let $N \in \mb{N}$, and let $t_1,\ldots,t_k$ be $k$-tuple of complex numbers in the unit disc $\mb{D}$. Put $m =(N+1)^k$. Then there exist $m$ points $\{w^i = (w_1^i, \ldots, w_k^i)\}_{i=1}^{m}$ on the $k$-torus $\mb{T}^k$, and $m$ nonnegative numbers $a_1 \ldots, a_{m}$ satisfying $\sum a_i = 1$, such that for every polynomial $p(z_1, \ldots, z_k)$ of degree less than or equal to  $N$, 
\be\label{eq:cubature}
p(t_1, \ldots, t_k) = \sum_{i=1}^{m} a_i p(w^i_1, \ldots, w^i_k) .
\ee
\end{corollary}
Given an integer $N$ and a point $(t_1, \ldots, t_k) \in \mb{D}^k$, the existence of a finite number of points $w^1, \ldots, w^m$ on the torus $\mb{T}^k$, together with a finite number of weights $a_1, \ldots, a_m$ such that (\ref{eq:cubature}) holds for all polynomials of degree less than or equal to $N$ can be obtained also by classical analytical means. One uses the Poisson integral on the torus together with a classical (and more general) result of Tchakaloff \cite{Tchakaloff} (see also \cite{Putinar}). Note, however, how the elementary proof that we gave above suggests explicitly a way to find the points $w^1, \ldots, w^m$ and the weights $a_1, \ldots, a_m$.

\section{Two commuting contractions: ``commutant lifting" and Ando's Theorem}

As we discussed in the previous section, three or more commuting contractions might not have a unitary dilation. In contrast, for two commuting contractions, there is the following theorem.
\begin{theorem}{\bf Ando's Dilation Theorem \cite{Ando}.}
Every pair of commuting contractions has a unitary dilation.
\end{theorem}
It follows that von Neumann's Inequality holds for two commuting contractions. This is sometimes referred to, and rightly so, as \emph{Ando's Inequality}. It would be interesting to provide a proof of Ando's Inequality for pairs of commuting contractions on a finite dimensional space using $N$-dilations, as we have done above for the single operator case. 

\vspace{0.2cm}
\noindent{\bf Problem C:} \emph{Given $N \in \mb{N}$, is it true that every pair of commuting contractions on $H$ has a unitary $N$-dilation on a finite dimensional space?}
\vspace{0.2cm}

Any solution to the above problem would be interesting.  If the answer is \emph{yes}, then this would show that the mysterious difference between two commuting contractions (Ando's Theorem) and three commuting contractions (Parrott's counter example) is already present at the finite dimensional level. On the other hand, Ando's Theorem is rather different from Sz.-Nagy's Theorem: the minimal dilation of a couple of commuting contractions is not unique, and the proof of Ando's Theorem is not canonical. We have heard an operator theorist say that Ando's Theorem makes operator theorists feel ``uneasy". If the answer to the above problem is \emph{no}, it would show that Ando's Theorem is a truly infinite dimensional phenomenon.

A close relative of Ando's Theorem is the following (which is usually \emph{not} called the Commutant Lifting Theorem, but is related. See \cite[Chapter VII]{FoiasFrazho}).
\begin{theorem}{\bf ``Commutant Lifting Theorem".}
Let $A$ and $B$ be be commuting contractions on $H$, and let $U$ on $K$ be the minimal unitary dilation of $A$. Then there exists a contraction $V$ on $K$ that commutes with $U$, such that
\bes
B^k = P_H V^k P_H \,\, , \,\, k \in \mb{N}.
\ees
\end{theorem}
The theorem that is usually referred to as the Commutant Lifting Theorem makes the same assertion as in the above theorem, with two differences. The first difference is that $U$ is taken to be the minimal \emph{isometric} dilation of $A$, and not the minimal \emph{unitary} dilation. The second difference is that $V$ is then asserted to be a \emph{lifting} of $B$, and not merely a dilation. This means that $B^*$ is equal to the restriction of $V^*$ to $H$.

Ando's Theorem and the Commutant Lifting Theorem are easily derived one from the other (see \cite[Section VII.6]{FoiasFrazho}). In the context of $N$-dilations, all we could show is that ``commutant lifting" implies the existence of a joint $N$-dilation.

\begin{proposition}
Let $A$ and $B$ be two commuting contractions on $H$. Let $U$ be a unitary $N$-dilation of $A$ on a finite dimensional space $K$, and assume that there is a contraction $V$ on $K$ that commutes with $U$, such that the pair $U,V$ is a (non-unitary) $N$-dilation of $A,B$. Then the pair $A,B$ has a unitary $N$-dilation on a finite dimensional space.
\end{proposition}
\begin{proof}
Since $U$ is normal and $V$ commutes with $U$, Fuglede's Theorem implies that $U$ and $V$ doubly commute. Thus, the technique used in the proof of Theorem \ref{thm:doublycommuting} provides us with a unitary $N$-dilation for $U,V$, which is the required unitary $N$-dilation for $A,B$.
\end{proof}
This leads us to the following problem:

\vspace{0.2cm}
\noindent{\bf Problem D:} \emph{Is there a  ``Commutant Lifting Theorem" in the setting of unitary $N$-dilations?}

\section{Dilations of completely positive maps}

There is a dilation theory of contractive completely positive maps that is reminiscent of the dilation theory of contractions. In this section we find that this theory is truly infinite dimensional in nature. 

A linear map $\phi : M_k(\mb{C}) \rightarrow M_k(\mb{C})$ is said to be \emph{positive} if $\phi(A) \geq 0$ whenever $A \geq 0$. A linear map $\phi : M_k(\mb{C}) \rightarrow M_k(\mb{C})$ is said to be \emph{completely positive} if, for every $m \in \mb{N}$, the linear map 
\bes
\phi \otimes I_m : M_k(\mb{C}) \otimes M_m(\mb{C}) = M_{km}(\mb{C}) \longrightarrow M_k(\mb{C}) \otimes M_m(\mb{C}) = M_{km}(\mb{C})
\ees
defined by $\phi \otimes I_m (A \otimes B) = \phi(A) \otimes B$, is positive. We will use the acronym CP for ``completely positive".

Let $H$ be a finite dimensional complex Hilbert space of dimension $n$. We identify the algebra $B(H)$ of linear operators on $H$ with $M_n(\mb{C})$ in the usual way. Recall that by Choi's Theorem  \cite[Theorem 1]{Choi}, the CP maps on $B(H) = M_n(\mb{C})$ are precisely those that can be represented in the form
\bes
\phi(T) = \sum_{i=1}^d A_i T A_i^*,
\ees
where $d \in \mb{N}$ and $A_1, \ldots, A_d \in B(H)$. The minimal $d$ for which such a representation is possible is called the \emph{index} of $\phi$. 

\begin{definition}
Let $\phi$ be a CP map acting on $B(H)$. A \emph{$*$-endomorphic dilation} for $\phi$  consists of a Hilbert space $K \supseteq H$ and a $*$-endomorphism $\alpha$ of $B(K)$ such that
\be\label{eq:edil}
\phi^k (T) = P_H \alpha^k (T) P_H
\ee
for all $k \in \mb{N}$ and all $T \in B(H)$.
\end{definition}
Note that in the above definition we identified $B(H)$ with the subalgebra of $B(K)$ given by  $P_H B(K)P_H \subseteq B(K)$. 
The following theorem is the analogue of Sz.-Nagy's Dilation Theorem appropriate for CP maps, and it is a key theorem in quantum dynamics.
\begin{theorem}\label{thm:Bhat}{\bf Bhat's Dilation Theorem\footnote{Bhat and the other researchers cited also prove this theorem for (the much harder case of) continuous one-parameter semigroups of CP maps.}} \emph{(see \cite{Bhat96,BS00,MS02,SeLegue}).} 
Every contractive CP map has a $*$-endomorphic dilation.
\end{theorem}
The relation between Bhat's and Sz.-Nagy's  Dilation Theorems is deep (see Section 5 in \cite{ShalitSolel}). In fact, the methods of \cite{MS02} and \cite{Shreprep} enable to prove Bhat's Theorem by reducing it to the situation in Sz.-Nagy's Theorem (see \cite[Corollary 6.3]{ShalitSolel}), and also \emph{vice versa}. With this in mind, what follows may seem a little mysterious.

As in the case of the unitary dilation, when $\phi$ is a CP map acting on $B(H)$ which is not a $*$-endomorphism, then the dilation $\alpha$ necessarily acts on $B(K)$ where $K$ is infinite dimensional. Note that the well known Stinespring dilation of $\phi$ involves only finite dimensional spaces when $\dim H < \infty$, but it does not qualify as a $*$-endomorphic dilation in the above sense. One might hope that, as in the case of unitary dilations, one can find dilations acting on finite dimensional spaces for which (\ref{eq:edil}) holds only for finitely many $k$. However, this is impossible. The following proposition should be considered as evidence that dilation theory of CP maps is far more delicate than the dilation theory of operators (see also Section 5.3 in \cite{ShalitSolel}). 

\begin{proposition}
Let $\phi$ be a CP map on $B(H)$ of index $d>1$. Suppose that $K \supseteq H$ and that $\alpha$ is a $*$-endomorphism on $B(K)$ such that (\ref{eq:edil}) holds for $k=1$. Then $K$ is infinite dimensional.
\end{proposition}
\begin{proof}
If $K$ is finite dimensional then $\alpha$ is a $*$-automorphism. Consequently, there is a unitary $U$ on $K$ such that 
\bes
\alpha(T) = UTU^* \quad , \quad T \in B(H).
\ees
From (\ref{eq:edil}) holding for $k=1$ we have 
\bes
\phi(T) = (P_HU P_H) T (P_H U P_H)^* \quad , \quad T \in B(H),
\ees
which shows that the index of $\phi$ is $1$.
\end{proof}
\begin{remark}\emph{
By \cite[Theorem 5]{Choi}, if a CP map $\phi$ has index $1$ then it is an extreme point of the convex set of CP maps $\psi$ on $B(H)$ such that $\psi(I) = \phi(I)$. Thus, ``most" CP maps have index bigger than $1$.}
\end{remark}

\subsection*{Acknowledgments} We would like to thank Ken Davidson, Adam Fuller, Rajesh Pereira and Franciszek Szafraniec for pointing our attention to some important references. The second author would like to thank Ken Davidson for the warm and generous hospitality provided at the Department of Pure Mathematics at the University of Waterloo.

\subsection*{Note added after acceptance} After this paper was completed problems {\bf B} and {\bf C} were solved in the affirmative; see J. E. M\raise.45ex\hbox{c}Carthy and O. M. Shalit, {\em Unitary N-dilations for tuples of commuting matrices}, preprint, arXiv:1105.2020v1.


\end{document}